\pdfoutput=1
\documentclass[12pt,reqno,final]{amsart}
\usepackage{amsmath,a4wide}
\usepackage{stmaryrd,mathrsfs,bm,amsthm,amssymb, xcolor,tikz}
\usetikzlibrary{arrows,shapes,patterns,calc,fadings,arrows.meta,decorations.pathmorphing,decorations.pathreplacing,decorations.markings}
\usepackage{hyperref}
\hypersetup{
	colorlinks = true,
	linkcolor = {blue},
	urlcolor = {red},
	citecolor = {blue}
}
\usepackage{enumerate}
\usepackage{enumitem}
\usepackage{comment}
\usepackage{listings}

\usepackage[nameinlink,capitalise,sort]{cleveref}
\crefname{equation}{}{} 
\crefname{enumi}{}{} 
\crefname{figure}{Figure}{Figures}

\begingroup
\newtheorem{theorem}{Theorem}[section]
\newtheorem{lemma}[theorem]{Lemma}
\newtheorem{proposition}[theorem]{Proposition}
\newtheorem{corollary}[theorem]{Corollary}
\endgroup

\theoremstyle{definition}
\newtheorem{definition}[theorem]{Definition}
\newtheorem{remark}[theorem]{Remark}

\DeclareMathOperator{\dist}{dist}

\def\Xint#1{\mathchoice
{\XXint\displaystyle\textstyle{#1}}%
{\XXint\textstyle\scriptstyle{#1}}%
{\XXint\scriptstyle\scriptscriptstyle{#1}}%
{\XXint\scriptscriptstyle\scriptscriptstyle{#1}}%
\!\int}
\def\XXint#1#2#3{{\setbox0=\hbox{$#1{#2#3}{\int}$ }
\vcenter{\hbox{$#2#3$ }}\kern-.6\wd0}}

\def\dashint{\Xint-}

\numberwithin{equation}{section}

\title{Caffarelli's work on elliptic free boundary problems}
\author[C. De Lellis]{Camillo De Lellis}
\address{School of Mathematics, Institute for Advanced Study, 1 Einstein Dr., Princeton NJ 08540, USA}
\email{camillo.delellis@ias.edu}

\begin{document}

\begin{abstract}
The aim of the note is to illustrate some of the ideas introduced by Luis Caffarelli in his groundbreaking works on the regularity theory for elliptic free boundary problems, in a way which can be understood by non-experts.
\end{abstract}

\maketitle

\tableofcontents

\section{Introduction}
The aim of this note is to give an idea of the profound impact of Luis Caffarelli's work on free boundary problems. I think that it is fair to say, without diminishing the importance of his contributions, that the impact has been felt mostly in certain type of free boundary problems, which are classified as ``elliptic and parabolic'' in the literature. While these are two class of problems which share certain commonality (as it is the case for linear PDEs: the Laplace and heat equations certainly have many points of contact), confronted with the first of many choices to make, I declare right away that I will limit myself to the elliptic side.

\medskip

I dare say that the mathematical theory of elliptic free boundary problems can be considered a standalone field in the area of partial differential equations. A quick search on Mathscinet, Zentralblatt, or arXiv of a few keywords will produce a number of research papers on the topic and I hope it will convince most people that my claim is not outlandish. It is nearly impossible to give a good account of the developments in such a field in a single paper. In fact there are books and survey articles (see e.g. \cite{Caffarelli-Salsa,Alessio,Bozho,Xavi}), forming a quite bulky set of references, and each of these references is devoted to some particular aspect or to some particular family of problems. 

The claim that ``elliptic free boundary problems form a field'' is in great part due to the revolutionary ideas injected by Luis since his very first papers on the topic, cf. \cite{Caffarelli-Acta}, and even giving an account of {\em all} the research produced by him on the subject is a demanding task. Making a somewhat limited choice of what to explain to the reader seemed unavoidable to me. On the other hand I will try to compensate such arbitrary and limited choice by keeping this article as nontechnical as possible, in the hope that it can be understood by a general audience of mathematicians. In order to follow the discussions in this note a basic knowledge of measure theory, of functional analysis, and of very elementary facts about harmonic functions will suffice. In fact I will list immediately all the facts about harmonic functions which I will take for granted:
\begin{itemize}
\item the mean value property (and the elementary estimates on higher derivatives which can be derived from it);
\item the maximum principle;
\item the Liouville theorem;
\item the Weyl lemma\footnote{By ``Weyl's lemma'' I understand the statement that a weakly harmonic function is in fact smooth and harmonic in the classical sense; a weakly harmonic function is in turn understood as an element $u$ of $H^1 (D)$ (where $D$ is some open set of $\mathbb R^n$) which is stationary for the Dirichlet energy, or in other words such that $\langle \nabla u, \nabla v\rangle_{L^2} =0$ for every $v\in C^\infty_c (D)$.}.
\end{itemize}
The classical (19-th century) Harnack estimate for harmonic functions will also play a pivotal role, but I will recall what it is at the appropriate place. Finally, I will assume that the reader has seen once in her life the classical computations to derive Euler-Lagrange conditions of minima in the calculus of variations.

\medskip

Let me close this introduction with an important disclaimer: I am an outsider in the subject, because I never worked on any free boundary problem. When Carlos (Kenig) asked me if I would entertain the idea of writing a piece for the Bulletin illustrating Luis' works in the field, I in fact objected that there are many researchers much more qualified than me to accomplish such a task. Carlos however countered that he was looking precisely for the point of view of a layman and then, perhaps too rashly, I agreed, tantalized by the idea of learning more about a work which is so widely (and justly) celebrated.

\subsection{First roadmap} The following is the rough plan of the note. 
\begin{itemize}
\item In Section \ref{s:uno} I will explain what free boundary problems are in general and I will introduce four prototypical problems in which the work of Luis changed the landscape. These problems have all a variational nature and I will explain why they can be regarded as ``free boundary problems''. I will also briefly mention why the existence of their solutions is, nowadays, a simple byproduct of standard functional analysis.  
\item In Section \ref{s:due} I will explain why the work of Luis is so groundbreaking: let me anticipate that prior to his papers almost nothing was known on the ``regularity of the free boundary'' of the solutions. With his works not only Luis produced theorems which lifted our almost empty knowledge to a very satisfactory level, but he also paved the way to a variety of later developments in the field. I will argue that there is an underlying philosophy to his approach to regularity which is quite general and powerful, and which is in fact quite neatly explained in other references, however all the ones I know are written for researchers in PDE.
\item In the Sections \ref{s:tre}, \ref{s:quattro}, \ref{s:cinque}, \ref{s:sei}, and \ref{s:sette} I will explain in some detail the celebrated regularity theorem of Luis for one of the four variational problems introduced in Section \ref{s:uno}. The idea is that the interested reader will see the philosophy mentioned above ``in action'' on a concrete example. The theorem also happens to be one of the very first contributions of Luis in the field and it can be considered as the origin of the whole story. Since at this point I lack the basic terminology to explain what each of these sections contains, a roadmap to them will be given in Section \ref{s:roadmap-2}.  
\item Section \ref{s:otto} will serve as a teaser for further readings. I will mention how similar regularity theorems were proved by Luis and some of his collaborators in the context of the other three problems and I will point out further readings in the area.   
\end{itemize}

\subsection{Acknowledgments} The author acknowledges the support of the NSF foundation through the grant DMS-2350252. He also expresses his gratitude to Guido De Philippis, Alessio Figalli, Carlos Kenig, and Luca Spolaor, for reading carefully earlier versions of the manuscript and helping him improving it in several ways. 

\section{Some elliptic free boundary problems}\label{s:uno}

Consider a (bounded and smooth) open set $D\subset \mathbb R^n$. In a free boundary problem one typically looks for a pair $(\Omega, u)$, where $\Omega \subset D$ is another (generally open) domain and $u: \Omega \to \mathbb R$ a function, subject to:
\begin{itemize}
\item a partial differential equation for $u$ in the unknown domain,
\begin{equation}\label{e:PDE}
Lu = 0 \qquad \mbox{in $\Omega$}\, ,
\end{equation}
where $L$ is some differential operator;
\item a boundary condition on the ``fixed boundary'', namely
\begin{equation}\label{e:BC1}
B_0 u = 0 \qquad \mbox{on $\partial D \cap \partial \Omega$} 
\end{equation}
(which typically would translate into a determined PDE problem if $D=\Omega$);
\item an ``overdetermined'' boundary condition on the ``free boundary'', namely
\begin{equation}\label{e:BC2}
B_1 u = 0 \qquad \mbox{on $D\cap \partial \Omega$.}
\end{equation}
\end{itemize}
As aleady mentioned $\Omega$ is also an unknown of the problem and we are ``free'' to set its boundary in $D$: if we were dealing with a fixed domain the overdetermined PDE problem would typically have no solutions, but because we are free to vary $\Omega$ we hope there is some special choice of it for which \eqref{e:PDE}, \eqref{e:BC1}, and \eqref{e:BC2} do admit a solution $u$. 

This setting is at the same time not general enough for the four examples that I will give below and too general for the purpose of this note. First of all, the nature of the PDE problem could be anything, given the generality above. However, the works which I am going to illustrate all end up being ``elliptic'', and in fact the reader can think of $Lu$ as the ordinary Laplacian $\Delta u = \sum_i \frac{\partial^2 u}{\partial x_i^2}$. Similarly the boundary conditions will satisfy some structural properties, in fact they will involve just the value of $u$ and its derivatives at the boundary: the reader can think that they take the form $F_i (x, u(x), \nabla u (x))=0$ for some suitable functions $F_i$. As already hinted in the introduction, Luis' work has also been deeply influential on parabolic problems, but in this note I will completely neglect them.

The setting above is however not general enough to include the second example given below, the ``thin obstacle''. In the latter the unknown domain is not really the place where the PDE is solved, but the place where certain boundary conditions hold. In other words, $\Omega$ is not unknown and is the whole domain $D$, whereas (besides the function $u$) the second unknown is a subset $\Omega'\subset \partial D$: it is rather the boundary condition which will differ on $\Omega'$ and on $\partial D \setminus \Omega'$. 

\medskip

We will now analyze four examples of variational problems whose minima satisfy elliptic free boundary problems. Each of them are motivated by actual problems outside of mathematics (in fact prototypical problems in several areas of mathematical physics) and have a rich history, but I will make no attempt to describe the latter and defer the reader to the literature in the bibliography.

\subsection{The obstacle problem}\label{s:obstacle}
The first example is the classical ``obstacle problem''. The data of the problem are:
\begin{itemize}
\item a domain $D\subset \mathbb R^n$ (again bounded and smooth), 
\item a smooth function $\varphi$ on it (the {\em obstacle}),
\item  and a smooth function $g$ on $\partial D$.
\end{itemize}
We then look for the minimum of a suitable energy functional $E(u)$ on the class of functions $u: D \to \mathbb R$ which ``lie over the obstacle'', i.e. such that $u\geq \varphi$ over the whole $D$, and satisfy the boundary condition $u=g$ on $\partial D$. As the two conditions would otherwise be incompatible, we will assume that $\varphi|_{\partial D} \leq g$. The model energy we will consider here is the Dirichlet energy $E (u) = \int_D |\nabla u|^2$, but of course one can consider more general ones. In this model example the existence of the solution is a simple exercise in functional analysis: if we consider the subset $K\subset H^1 (D)$ of functions satisfying the two constraints, it should be pretty obvious to the reader that $K$ is convex and it is not too hard to show that it is closed. The energy being itself strictly convex, so not only the minimizer exists in $K$, but it is in fact unique (see for instance \cite[Theorem 5.6]{Brezis}). 

However, why is the minimum just found also the solution of a ``free boundary problem''? Let me give an answer under some simplifying assumptions. If $u$ is continuous and we set $\Omega:= \{u>\varphi\}$, clearly $\Delta u =0$ on $\Omega$: $u+\varepsilon w$ lies in $K$ whenever $w \in C^\infty_c (\Omega)$ and $\varepsilon$ is small enough, leading to the first variation condition
\[
\int \nabla u \nabla w = 0 \qquad \forall w\in C^\infty_c (\Omega)\, ,
\]
so Weyl's lemma does the rest.

At the fixed boundary $\partial D\subset \partial \Omega$ $u$ satisfies the Dirichlet boundary condition $u|_{\partial D}=g$. The ``free boundary'' is now $\partial \Omega \cap D$. Clearly by definition 
\begin{equation}\label{e:Dirichlet}
u|_{D\cap \partial \Omega} = \varphi\, ,
\end{equation}
which is one boundary condition. If $u$ were differentiable at the free boundary, we would also have the additional condition $\nabla u|_{D \cap \partial \Omega} = \nabla \varphi$ (which is really a condition on the normal derivative when the free boundary $D\cap \partial \Omega$ is smooth, because the equality of tangential derivatives is implied by \eqref{e:Dirichlet}), and the reader will recognize the ``overdetermined boundary condition'' at the free boundary: we have a Dirichlet and a Neumann conditions at the same time. However, is it reasonable to expect that the solution is differentiable at the free boundary? 

\begin{remark}\label{r:notC2}
Note that the question is quite well motivated. In fact it is clear that the second derivatives cannot, in general, be continuous: on the one hand we know that on $\Omega$ the trace of $D^2 u$ is zero, on the other hand, if the complement has nonempty interior then the trace of $D^2 u$ on it equals the trace of $D^2 \varphi$: if $\varphi$ is, say, concave (a reasonable assumption for an obstacle), $D^2 u$ must undergo a discontinuity at $\partial \Omega\cap D$. In order to get a concrete example we just need to find a strictly convex obstacle $\varphi$ and a boundary condition $g$ which ``forces'' the minimizer to touch the obstacle on some set with nonempty interior. This is, however, a rather simple task.
\end{remark}  

It is on the other hand not difficult to convince oneself that the continuity of the first derivative of the solution at the free boundary is a reasonable ``variational restriction'' for a minimizer. Assume for instance that $\partial \Omega$ is smooth at some point $x_0$. $u$ is then smooth on ``both sides'' of $\partial \Omega$ in a neighborhood of $x_0$: on $\Omega$ (say the right side) it is an harmonic function satisfying a Dirichlet boundary condition at $\partial \Omega$, while on the other (the left) side it coincides with the smooth function $\varphi$. It is thus differentiable from both sides at $x_0$ and it suffices to show that the normal derivative does not jump. Denote by $\nu$ the unit normal that points towards the interior of $\Omega$. The normal derivative $\frac{\partial u^-}{\partial \nu}$ on the left side equals $\frac{\partial \varphi}{\partial \nu}$. On the right side we must have $\frac{\partial u^+}{\partial \nu}\geq \frac{\partial \varphi}{\partial \nu}$, because of the condition $u\geq \varphi$. On the other hand it is energetically not convenient for $u$ to satisfy a strict inequality: the reader is invited to analyze the model situation of $1$ dimension: $u$ is linear on both sides: ``cutting the corner'' as in Figure \ref{f:cutting} will lower its energy while keeping the constraint.  

\begin{figure}[htbp]
\begin{center}
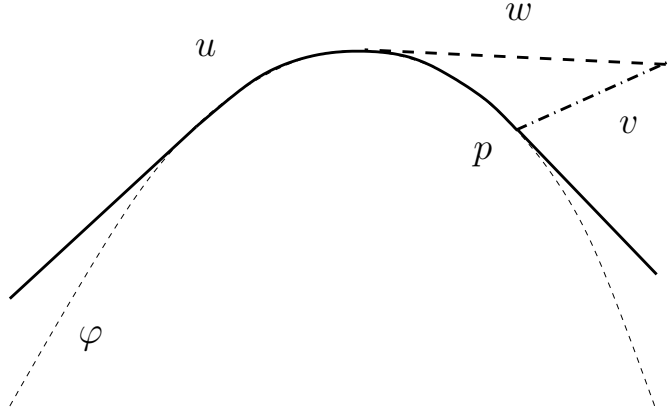
\end{center}
\caption{A visual explanation of why the solution of the obstacle problem cannot leave the obstacle forming a corner, in a simple one-dimensional case. The thin dashed line is the graph of the obstacle $\varphi$. The thick line is the graph of a solution of some obstacle problem $u$ for a boundary data $g>\varphi$ (the value of $u$ at the two endpoints). $v$ is an hypothetical solution of some other obstacle problem which agrees with $u$ on the left of the point $p$ but then leaves the obstacle forming an angle (the dotted-dashed thick line; clearly the boundary data for $v$ differs at the right endpoint). However it is clear that this behavior is energetically unfavorable: the dashed thick line is the graph of a competitor $w$ which achieves the same boundary condition as $v$ but has lower energy.}\label{fig:1}\label{f:cutting}
\end{figure}

\subsection{The thin obstacle} In the thin obstacle problem we again wish to minimize an elliptic energy (which for simplicity we assume again to be the Dirichlet energy) over a class of functions which this time are assumed to satisfy a boundary condition of Dirichlet type $u=g$ on some portion of the boundary and an inequality of the form $u\geq \varphi$ on the remaining portion. A model situation is the following:
\begin{itemize}
\item $D$ is the half-ball $B^+= \{x\in \mathbb R^n : |x|<1, x_n >0\}$,
\item $u=g$ on the half sphere $S^+ = \partial B^+\cap \{x_n>0\}$,
\item and $u\geq \varphi$ on the flat part of the boundary  $H^+:= \partial B^+\cap \{x_n =0\}$. 
\end{itemize}
Once again it makes sense to assume the compatibility condition $g\geq \varphi$ on $H^+\cap S^+$ and we notice that the existence (and uniqueness) of the solution is a simple byproduct of the convexities of the problem (the energy is strictly convex and the appropriate functional-analytic domain is convex). 

The solution does not fall in the ``classical free boundary formulation'' explained above, as I already anticipated to the reader: it is in fact an harmonic function on the whole domain $B^+$, because any variation with compact support in $B^+$ does not touch the conditions at the boundary. On the other hand we can understand the free boundary as some appropriate (presumably $(n-2)$-dimensional) set in $H^+$ which subdivides it into two regions where $u$ satifies two different boundary conditions. In fact, assume that $u$ is continuous up to the boundary, so that $H^+\cap \{u>\varphi\}$ is relatively open in $H^+$. We are then free to vary the function by adding $\varepsilon w$ for some $w$ which does not vanish on $H^+ \cap \{u>\varphi\}$, provided that it vanishes on $S^+$ and on $H^+ \cap \{u=\varphi\}$. With this choice of variation it is a classical computation that $u$ must satisfy the Neumann boundary condition $\frac{\partial u}{\partial \nu}=0$ on $H^+\cap \{u> \varphi\}$. In the relative interior of $H^+\cap \{u=\varphi\}$ clearly $u$ satisfies the Dirichlet boundary condition $u=\varphi$ by definition, but it is also natural to expect that it must satisfy $\frac{\partial u}{\partial \nu}\leq 0$ (where $\nu$ is pointing in the direction of $B^+$), since it is clearly energetically unfavorable for the function $u$ to go ``upwards'' while entering the domain (in that case lowering the slope clearly lowers the energy). 

Hence, if we set $\Omega':= H^+ \cap \{u>\varphi\}$ we end up with the PDE problem 
\begin{equation}
\left\{
\begin{array}{ll}
\Delta u = 0 \qquad & \mbox{on $B^+$}\\
u = g \qquad &\mbox{on $S^+$}\\
u-\varphi>0, {\textstyle{\frac{\partial u}{\partial \nu}}}=0 \qquad &\mbox{on $\Omega'$}\\
u-\varphi=0,  {\textstyle{\frac{\partial u}{\partial \nu}}}\leq 0 \qquad &\mbox{on $H^+\setminus \Omega'$}
\end{array}
\right.
\end{equation}
in the unknown pair $(\Omega', u)$. 

\subsection{The Bernoulli one-phase problem} In the third example, having fixed a domain $D$ and {\em nonnegative} function $g$ on $\partial D$, we turn to minimizing the energy 
\[
E (u) := \int |\nabla u|^2 + \Lambda |\{u>0\}|\, 
\]
among all functions $u$ such that $u|_{\partial D}=g$, where $|A|$ denotes the Lebesgue measure of any (measurable) set $A$ and $\Lambda$ is a positive constant. 

The existence of a minimizer is less obvious than in the previous cases. Observe that we could rewrite the energy as 
\[
E (u) = \int (|\nabla u|^2 + \chi (u))
\] 
introducing the step function 
\[
\chi (w) := \left\{
\begin{array}{ll}
\Lambda \qquad &\mbox{if $w>0$}\\
0 \qquad \qquad &\mbox{otherwise}\, .
\end{array}
\right.
\]
Let us approximate $\chi$ pointwise monotonically with a sequence of increasing smooth nonnegative functions $\chi_k\uparrow \chi$. The functionals 
\[
u\mapsto \int \chi_k (u)
\]
are then continuous for the weak topology over bounded subsets of $H^1$. Therefore
\[
u\mapsto \int \chi (u) = \sup_k \int \chi_k (u)\, 
\]
is lower semicontinuous. The energy $E(u)$ is thus lower semicontinuous as well and the existence of a minimizer can be concluded from the direct methods in the calculus of variations. 

Once again we address the question: why is a minimizer a solution of an (elliptic) free boundary problem? Let us assume that $u$ is continuous and introduce $\Omega:= \{u>0\}$. As in the previous cases $\Delta u =0$ on $\Omega$ and $u=g$ on $\partial D \cap \partial \Omega$. Moreover, $u=0$ on $D\cap \partial \Omega$ by definition of $\Omega$. 
According to our setting, we need to find a second boundary conditions at the free boundary. 

Consider then a smooth vector field $X\in C^\infty_c (D, \mathbb R^n)$ and let $\Phi (t,\cdot)$ be the one-parameter family of diffeomorphisms of $D$ generated by it. If we let $u_\varepsilon (x) := u (\Phi (\varepsilon, x))$, we see immediately that $u_\varepsilon (x) = u(x) = g (x)$ at the boundary of $D$. We can thus conclude the Euler-Lagrange condition
\begin{equation}\label{e:internal_variation}
\left.\frac{d}{d\varepsilon}\right|_{\varepsilon =0} E (u_\varepsilon) = 0\, .
\end{equation}
If we assume that the free boundary $D\cap \partial \Omega$ is smooth and denote by $\nu$ the unit normal to it pointing towards $\Omega$, an elementary computation shows that the left hand side of \eqref{e:internal_variation} equals the boundary integral
\[
\int_{D\cap \partial \Omega} (\Lambda - |\nabla u|^2) X \cdot \nu\, .
\]
Given that $X$ is an arbitrary vector field, we infer the boundary condition $|\nabla u|=\sqrt{\Lambda}$ at $D\cap \partial \Omega$.

\subsection{The two-phase problem} In the two-phase problem the set-up is similar to that of the Bernoulli problem, except that $g$ is not necessarily assumed to have a sign and the energy to minimize is given by
\[
\int |\nabla u|^2 + \Lambda^+ |\{u>0\}| + \Lambda^- |\{u<0\}|\, ,
\] 
where $\Lambda^+$ and $\Lambda^-$ are positive and distinct. 
The existence of a minimizer can be proved following the same argument given for the one-phase problem. As in the one-phase problem, if we assume $u$ continuous, it is obvious that $\Delta u =0$ on $\Omega := \{u\neq 0\}$. The conditions satisfied by the free boundary are however more complicated. 

First of all, if in a neighborhood of a point $x_0\in D\cap \partial \Omega$ the function is nonpositive or nonnegative, in that neighborhood the function is also a minimizer of a one-phase functional. Such points are therefore named ``one-phase'' points and the boundary conditions on them match the one derived in the previous section. The new interesting part is thus given by the ``two-phase'' points $x_0$, namely those which belong to $D\cap \partial \{u>0\} \cap \partial \{u<0\}$. Under the simplifying assumption that $\partial \{u>0\}\cap \{\partial u<0\}$ is a smooth interface and that $u$ is differentiable on ``both sides'' of it, the same variation used in the one-phase case leads to the boundary condition 
\[
|\nabla u^+|^2 - |\nabla u^-|^2 = \Lambda^+- \Lambda^-\, .
\]

\section{Regularity of the free boundary: general themes}\label{s:due}

In all four examples of the previous section, deriving the existence of a minimizer of the corresponding variational problem is pretty straightforward (while, of course, we are taking advantage of a long history of functional analytic tools developed in the past decades). However the resulting minimizer may be far from the description given at the beginning of the section of a solution of a ``free boundary problem''. 

To begin with, the only immediate information about the minimizers is that they belong to the space $H^1 (D)$. The very first question is thus whether we can gain some more regularity for the solution as {\em a function defined on the whole domain $D$}. Note however that we are not expecting, in general, too much regularity in this respect. We already argued for the obstacle problem that it is unreasonable to expect that a solution has continuous second derivatives across the free boundary, cf. Remark \ref{r:notC2}: a good guess would be that, as a function over the whole domain $D$, the minimizer is differentiable with Lipschitz continuous first derivative (in the PDE literature these functions form the class $C^{1,1}$). Likewise, if we take a closer look at a solution of the one-phase Bernoulli problem, we expect it to be identically $0$ on one side of the free boundary and to start off with positive normal derivative (in fact the absolute value of the latter should be $\sqrt{\Lambda}$) on the other side. We can therefore expect them to be, say, Lipschitz continuous over $D$, but not everywhere differentiable.  

Some results showing the expected regularity for the function $u$ when regarded as a function on the whole domain $D$ existed before Luis' work, for instance the $C^{1,1}$ regularity of solutions to the obstacle problem was indeed proved by Frehse, cf. \cite{Frehse}. For the other three problems mentioned in the previous sections later works of Luis (some with co-authors) got the expected regularity, cf. \cite{Caffarelli-79,Alt-Caffarelli,ACF}. 

Though interesting, fundamental, (and in the appropriate sense optimal) these results shed however little light on the issue which instead seems of primary interest, namely how regular is the ``free boundary''. Due to the results on the regularity as a function over $D$ we can however at least identify a good candidate for the latter object: for instance in the obstacle problem just the continuity of the solution (and Frehse's theorem goes well beyond that) allows us to introduce the open domain $\Omega:=\{u> \varphi\}$ and we can expect that its boundary will provide a ``classical'' solution of the free boundary problem in its ``PDE formulation''. On the other hand, for all we know at the moment, $\partial \Omega$ could very well be a fractal, or even a set of positive Lebesgue measure. 

Prior to Luis' first works in the area (i.e. essentially prior to \cite{Caffarelli-Acta}), the overall results in gaining some form of regularity for $\partial \Omega$ were almost inexistent: the ones I am aware of concern the first  two of the four examples in Section \ref{s:uno}: (notable) works of Hans Lewy, cf. \cite{Lewy1,Lewy2}, in the cases of the obstacle and thin obstacle problems proved regularity for the free boundary in the $2$-dimensional setting. There were instead quite general works showing that some initial regularity of $\partial \Omega$ could be bootstrapped to higher one (cf. the work \cite{KN}). 

The groundbreaking paper \cite{Caffarelli-Acta} not only gave a satisfactory answer to the boundary regularity of $\partial \Omega$ for the obstacle problem, but it also outlined a general strategy to attack the question which was later implemented by Luis himself and co-authors (e.g. in the other three examples of the previous section) and by other mathematicians in several other settings. The work outlined a general philosophy, which I will try to explain next.

\medskip

As already remarked, a first regularity result for the solution $u$ on the ``whole domain $D$'' allows us at least give sense to the open domain $\Omega$ whose boundary is supposed to be the ``classical free boundary''. Assume now for a moment that we knew the desired regularity (say $C^1$) of $D\cap \partial \Omega$ at some point $x_0\in D\cap \partial \Omega$, which for the sake of our discussion we can assume it is the origin. Then $\partial \Omega$ would have a tangent plane $H$ at $0$ and, up to applying a rotation, let us assume that it is $H=\{x_n=0\}$, with $\nu = (0,0, \ldots, 1)$ being the unit normal pointing towards $\Omega$. 

In the case of the Bernoulli one-phase problem we would know that\footnote{In principle the boundary condition derived in the previous section would tell us that $\frac{\partial u}{\partial \nu} (0) = \pm \sqrt{\Lambda}$, observe however that we can exclude the minus sign because the solution $u$ of our variational problem is everywhere nonnegative.} $\frac{\partial u}{\partial \nu} (0) = \sqrt{\Lambda}$, so we would infer that the rescalings
\[
u_r (x) := \frac{u(rx)}{r} 
\]   
converge uniformly on compact sets to the function
\[
u_0 (x) := \left\{
\begin{array}{ll}
0 \qquad & \mbox{if $x_n<0$}\\
\sqrt{\Lambda}\, x_n &\mbox{if $x_n>0$.}
\end{array}
\right.
\]
In the case of the obstacle problem, consider the function $v:= u-\varphi$, and recall that $\Omega$ is defined to be $\{u>\varphi\}= \{v>0\}$. This time observe however that the function is continuonsly differentiable, in particular $\nabla u (0)=\nabla \varphi (0)$ and hence $v$ grows no faster than quadratically ``away from the obstacle'' in the $x_n$ direction (in fact this is guaranteed by Frehse's theorem!). Consider that $\Delta v = - \Delta \varphi$ on $\Omega$: if we set $\lambda:= - \frac{1}{2} \Delta \varphi (0)$, we would conclude that the rescalings
\[
v_r (x) := \frac{v(rx)}{r^2} 
\]   
converge uniformly on compact sets to the function
\[
v_0 (x) := \left\{
\begin{array}{ll}
0 \qquad & \mbox{if $x_n<0$}\\
\lambda x_n^2 &\mbox{if $x_n>0$.}
\end{array}
\right.
\]
\begin{remark}\label{r:concavity}
Observe in passing an interesting fact: since $v>0$ on $\Omega$, the latter picture forces $\Delta \varphi (0)\leq 0$. In other words at a point where $\partial \Omega$ is regular, it cannot be that $\Delta \varphi$ is positive. 
\end{remark} 

The first idea of \cite{Caffarelli-Acta} is that, even before knowing any regularity for the free boundary, at any given point $x_0\in D \cap \partial \Omega$ it is anyhow useful to look at suitable rescalings and normalizations of the solution $u$ (or of some other appropriate new unknown, like $v=u-\varphi$), dictated by the intuition above. Even with no further information, the variational nature of the problems (or anyway suitable PDE estimates) gives at least enough compactness to take sequential limits of the $u_r$ (or of the $v_r$). For instance, in the case of the obstacle problem this is a simple outcome of Frehse's theorem which gives a uniform bound on the derivative of the rescaled functions (and thus uniform subsequential limits can be extracted using the Ascoli-Arzel\'a Theorem). The latter limits are called ``blow-ups'' in the literature. Notice an important point: since we are using a compactness argument, there is no guarantee that a unique limit exists at each given point (i.e. that the {\em whole family} $\{u_r\}_r$ has a unique limit as $r\downarrow 0$). 

A first important task is however to ensure that these objects are nontrivial. This is accomplished by proving suitable ``growth estimates'' from below as the function leaves the free boundary. Obviously these bounds from below must match, at the level of growth, the one which is dictated by the scaling.   

\medskip

Next, the discussion above also points out that, in order to have some hope of proving that $\partial \Omega$ is regular, we need such blow-ups to have a particular form, roughly speaking to depend only upon one variable, or said otherwise, being translation invariant along all directions parallel to some hyperplane $H$. The latter is then a good candidate (in fact the only candidate!) to be the tangent plane to $\partial \Omega$ at the given point. This is however too much to expect. It is not difficult to give examples where not all the points of $D\cap \partial \Omega$ have the latter property. However we can subdivide $D\cap \partial \Omega$ into two parts, a ``regular part'', and a ``singular part'': for a point to belong to the regular part we ask that at least one blow-up is the expected $1$-dimensional model (e.g. the ones analyzed above for the cases of the osbtacle problem and of the one-phase problem). The singular part is all that remains. 

At this point a pivotal and very much desired property is to infer some restriction about the type of blow-ups which we might encounter. The reader will notice that, by the very way we got to these blow-ups, they are presumably ``global'' solutions of the corresponding variational/PDE problem, i.e. they are defined on the whole space and their restrictions to compact sets are minimizers of the same variational problem. We might then expect some ``Liouville-type'' theorem which classifies these objects. However, Liouville type theorems are pretty hard to get. In the various examples given in this note, some extra help comes from additional information that can be inferred from the blow-up procedure. In the case of the obstacle problem this is for instance a powerful convexity estimate. In other cases an underlying monotonicity formula gives, for the blow-up, the constancy of an additional quantity\footnote{Often these tools can indeed be used to infer a Liouville type theorem without knowing that the objects classified came as a ``blow-up'' limit, but we will disregard this aspect.}. Notably, an interesting monotonicity formula was discovered later for the obstacle problem as well by Weiss (see \cite{Alessio} and the references therein), though it is not powerful enough to give, alone, a complete classification of blow-ups. 

At this point, however, the regular part is regular in name only: we have yet no reason to believe that it is better than a measurable set. A fundamenal discovery of \cite{Caffarelli-Acta} for the obstacle problem (which was then confirmed also in the other examples, although the proofs differ from case to case) is that there is an $\varepsilon$-regularity theorem: once at some scale the solution is sufficiently close to (one of) the good one-dimensional models, then it is, in some appropriate sense, a (relatively) smooth deformation of them.

In the various problems mentioned above, the proofs of these $\varepsilon$-regularity theorems differ substantially. In \cite{Caffarelli-Acta} Harnack estimates for elliptic PDEs do the lion share of the work (together with some very clever geometric considerations), but I will not enter into the topic now because this is the case which will occupy the next sections: we will give quite a few details about this.

\medskip

In many other situations the pivot is a ``flatness improving property'', which, roughly speaking, says that, if at a certain scale $u_r$ is sufficiently close to some ``one-dimensional model'' (or in other words is almost invariant along directions belonging to some hyperplane $H$), at a fixed smaller scale $u_{\beta r}$ is in fact even closer to some other ``one-dimensional model'' (i.e. it is closer to be invariant along directions belonging to some other hyperplane $H'$). The smaller scale $\beta$ is a fixed positive constant, while the improvement is by some fixed factor $\gamma$ (say $\frac{1}{2})$ in an appropriate metric (for instance the $C^0$ distance between the rescaled function and the model). 

This flatness improvement implies that, once at some scale the distance between the rescaling $u_{r_0}$ and our singularity model falls below some threshold, at all smaller scales the $\{u_r\}_{r<r_0}$ are in fact even closer, and form a Cauchy family in the chosen distance. We thus immediately infer from this flatness improving property that the blow-up limit is unique. It tells however more, as one can estimate the distance between the final blow-up and $u_{r_0}$ with a suitable power of $r_0$ (the exponent will depend on the decay factor $\gamma$ and on the smaller scale $\beta$). This triggers therefore a better property then just the $C^1$ regularity of $\partial \Omega$, it in fact implies an H\"older modulus of continuity of the tangent hyperplane. 

\medskip

Coming back to the splitting of the free boundary into a regular and singular part, an important point which is left out from the above discussion is the question of how large the singular set can be and what properties it has. This has been the subject of extensive studies in the 48 years that elapsed between Luis' foundation paper \cite{Caffarelli-Acta}, with the first important results due to Luis himself, both in the obstacle problem and in its cousins. However the latter is a story that goes beyond the scope of my note and I leave the reader to the various  
references suggested in Section \ref{s:otto}.

\medskip

Before getting further, a remark is in order. The pattern discovered by Luis in free boundary problems shares many similarities with several other fundamental discoveries of similar nature in the calculus of variations, analysis of PDEs, and geometric analysis. A pioneering example, in fact probably the very first example in the literature, is given by the works of De Giorgi, Reifenberg, Almgren, and Allard in the context of minimal surfaces. In various articles and surveys Luis himself acknowledges the profound impact that the latter literature, and especially De Giorgi's works in elliptic regularity theory, has had on him. However, reading the original paper \cite{Caffarelli-Acta} I came to the conclusion that the first work of Luis on the obstacle problem has a number of fundamental differences with the approach of the latter authors to regularity, and that the (fruitful!) connection with the minimal surface theory came rather in the later works. 

\subsection{Second roadmap}\label{s:roadmap-2} As claimed in the introduction, in the next five sections I will try to give, in as simple terms as possible, some details on the implementation of the above plan in the model case of the obstacle problem, following essentially \cite{Caffarelli-Acta}, but taking very much advantage of the beautiful survey \cite{Alessio}. In Section \ref{s:tre} I will give the very first elementary remark, which will serve as starting point of our discussion, while in Section \ref{s:quattro} I will discuss Frehse's $C^{1,1}$ regularity theorem, which we will interpret as a suitable bound from above on the growth of $v=u-\varphi$ as we get away from the free boundary. These sections describe essentially the state of the art of the regularity theory for the obstacle problem when \cite{Caffarelli-Acta} appeared, with the notable exception of the two-dimensional case. With the remaining three sections we enter into the fundamental contributions of Luis' paper. Section \ref{s:cinque} will discuss the matching lower bound, introduce the rescalings and their limits (the blow-ups). Section \ref{s:sei} will illustrate the fundamental convexity estimate which triggers the ``classification'' of the blow-ups and underlies the partitioning of the free boundary in a regular and singular part. Section \ref{s:sette} will finally outline the ideas behind the $\varepsilon$-regularity statement. In both these sections we will see the masterful use of the Harnack estimate for harmonic functions (and more in general for solution of elliptic PDEs) which underlie Luis' regularity theory in \cite{Caffarelli-Acta}. 

\section{The obstacle problem: superharmonicity and consequences}\label{s:tre}

In this and in the next section, we assume that $D\subset \mathbb R^n$ is a good domain (open, bounded, and sufficiently smooth), and $g: \partial D\to \mathbb R$ and $\varphi: \overline{D} \to \mathbb R$ two smooth functions with $\varphi|_{\partial D} < g$. We then denote by $u\in H^1 (D)$ the unique solution of the corresponding obstacle problem described in Section \ref{s:obstacle}, namely the unique element of the closed convex subset $K:= \{w\in H^1 (D): u\geq \varphi, u|_{\partial D} =g\} \subset H^1 (D)$ which minimizes the Dirichlet energy $E (u):= \int |\nabla u|^2$. 

The first important remark is that we can always compare the energy of $u$ with that of $u_\varepsilon := u+\varepsilon \psi$ if $\psi\in C^\infty_c (D)$ is a nonnegative function and $\varepsilon\geq 0$, given that under these assumptions $u+\varepsilon \psi \in K$. The condition $E (u_\varepsilon)\geq E(u)$ is then equivalent to 
\[
2\varepsilon \int \nabla u \cdot \nabla \psi + \varepsilon^2 \int |\nabla \psi|^2 \geq 0\, .
\] 
Dividing by $\varepsilon$ and letting $\varepsilon\downarrow 0$ we then infer the Euler-Lagrange condition
\begin{equation}\label{e:weakly_s}
\int \nabla u \cdot \nabla \psi \geq 0 \qquad \mbox{for all $\psi\in C^\infty_c (D)$ such that $\psi\geq 0$\, .}
\end{equation}
If $u$ were sufficiently smooth we could integrate by parts and conclude $-\Delta u \geq 0$, a condition which is commonly called ``superharmonicity'' in the literature. For $u$ smooth a simple computation shows that 
\[
\frac{d}{dr} \dashint_{\partial B_r (x)} u  = \frac{1}{\sigma_{n-1} r^{n-1}} \int_{B_r (x)} \Delta u\, ,
\]
where we denote by $\dashint$ the average and by $\sigma_{n-1}$ the $(n-1)$-dimensional measure of the unit sphere. Hence for smooth superharmonic functions both 
\[
r\mapsto \dashint_{\partial B_r (x)} u \qquad \mbox{and} \qquad r\mapsto \dashint_{B_r (x)} u
\]
are nonincreasing functions. This property is easily passed by approximation to ``weakly superharmonic'' functions (i.e. functions which satisfy \eqref{e:weakly_s}): it suffices to use a standard smoothing of such functions by convolution with a nonnegative smooth and radial kernel, which is seen to preserve \eqref{e:weakly_s}. 

This monotonicity property of superharmonic functions allows us to define the value of $u$ at every $x$ as 
\[
u (x) = \lim_{r\downarrow 0} \dashint_{B_r (x)} u\, . 
\]
Moreover, if we fix $r$ and consider the function
\[
x\mapsto \dashint_{B_r (x)} u\, ,
\]
the latter is easily seen to be continuous as soon as $u$ is summable. Thus the pointwise value of $u$ given above turns out to be the supremum of continuous functions, which means that it is also lower semicontinuous. But then the same property is shared by $v=u-\varphi$ and thus the set
\[
\Omega := \{u>\varphi\} = \{v>0\}\, 
\]
is an open set. Again by lower semicontinuity, the minimum of $u-\varphi$ over any compact subset of $\Omega$ is achieved and it is positive, which allows us to use the argument of Section \ref{s:obstacle} (i.e. to perturb the function with $\varepsilon \psi$ for any arbitrary $\psi\in C^\infty_c (\Omega)$ provided $\varepsilon$ is sufficiently small) and thus to show that $u$ is a classical harmonic function on $\Omega$ using Weyl's lemma. 

We gained therefore the starting point of our discussion: we have a candidate pair $(\Omega, u)$ to be a ``classical PDE solution'' of the free boundary problem discussed in Section \ref{s:obstacle} and in the next three sections we will see to which extent this guess is correct.  

\section{Frehse's regularity theorem for the solution}\label{s:quattro}

From here on we will mainly work with the function $v= u-\varphi$, which was already introduced in the previous two sections. Frehse's theorem states that, if $\varphi\in C^{1,1}$ (namely it has Lipschitz first derivatives), then $u\in C^{1,1}$ as well. To keep technicalities at a minimum I will however assume from here on that $\varphi$ is in fact twice differentiable with continuous second derivatives (namely $C^2$, in the PDE jargon). The first part of the proof is the most important one: we will accomplish a suitable upper bound on the ``growth'' of $v$ at points $x\in D\cap \partial \Omega$. We will then use classical estimates for harmonic functions to conclude.

\subsection{Growth bound} The aim is to show that $v$ grows at most quadratically away from $\partial \Omega$. Fix $x\in \Omega$ and let $x_0\in \partial \Omega$ be the closest point to it. We can assume that $x_0\in D\cap \partial \Omega$ simply by considering $x_0$ closer to $\partial \Omega$ than to $\partial D$ and let $r=|x-x_0| = \dist (x, \partial \Omega)$.  
Recall that $v= u-\varphi$, with $u$ harmonic in $B_r (x)$ and $\varphi\in C^2$: by the mean value property we have 
\[
u (x) = \dashint_{B_r (x)} u\, , 
\] 
while the $C^2$ regularity of $\varphi$ implies immediately
\[
-\varphi (x) \leq - \dashint_{B_r (x)} \varphi + C r^2\, .
\]
Combining the two inequalities we reach 
\begin{equation}\label{e:ub}
v (x) \leq \dashint_{B_r (x)} v + C r^2\, .
\end{equation}
On the other hand, $v$ is nonnegative, so 
\begin{equation}\label{e:ub2}
\dashint_{B_r (x)} v \leq 2^n \dashint_{B_{2r} (x_0)} v\, . 
\end{equation}
By the superharmonicity of $u$,
\[
\dashint_{B_{2r} (x_0)} u \leq u (x_0) = \varphi (x_0)\, ,
\]
while again because of the $C^2$ regularity of $\varphi$
\[
- \dashint_{B_{2r} (x_0)} \varphi \leq - \varphi (x_0) + Cr^2\, .
\]
Summing the last two inequalities we reach 
\begin{equation}\label{e:ub3}
\dashint_{B_{2r} (x_0)} v \leq C r^2\, .
\end{equation}
But then \eqref{e:ub}, \eqref{e:ub2}, and \eqref{e:ub3} together give 
\begin{equation}\label{e:growth_ub}
v (x) \leq C r^2 = C \dist (x, \partial \Omega)^2\, .
\end{equation}

\subsection{Interior $C^{1,1}$ regularity} We recall here a simple consequence of the mean-value property for harmonic functions, namely the estimate
\[
|D^k u (x)| \leq C r^{-k} \sup_{B_r (x)} |u|\, .
\]
Cosider thus a point $x\in \Omega$ closer to $D\cap \partial \Omega$ than to $\partial D$, and set $r:= \dist (x, \partial \Omega)$. The function $y\mapsto w (y) = u (y) - \varphi (x) - \nabla \varphi (x)\cdot (y-x)$ is harmonic in $B_r (x)$ and we can apply the above estimate with $k=2$ to infer
\[
|D^2 u (x)| = |D^2 w (x)| \leq C r^{-2} \sup_{B_r (x)} |w|\, .
\]
We can now use the Taylor formula for the $C^2$ function $\varphi$ to infer that 
\[
|w (y)|\leq |v(y)| + C r^2 \qquad \forall y\in B_r (x)\, .
\]
On the other hand the growth bound of the previous section gives $|v(y)|\leq Cr^2$ as well, so we infer that 
\[
|D^2 u (x)| \leq C\, .
\]
It is not difficult to see that this argument gives a uniform bound for the second derivative of $u$, and hence of the second derivative of $v$, over $\Omega\cap K$ for every compact $K\subset D$. Because $v$ vanishes identically on $K\setminus \Omega$, we could then expect that this gives the desired $C^{1,1}$ regularity. The argument is however slightly more subtle. Rather than giving all the details, we just remark that the idea above applied to the first derivative rather than the second gives a bound of type 
\[
|\nabla v (x)|\leq C \dist (x, \partial \Omega)\, .
\]
Therefore both $v$ and $\nabla v$ extend continuously to zero from $\Omega$ to $D$, thus proving the continuous differentiability of $v$. For the Lipschitz regularity of $\nabla v$ one has to work slightly more, but essentially it is a matter of integrating the bound for the second derivatives on the right paths.

\section{Lower bound on the growth and blow-ups}\label{s:cinque}

What was explained so far was known at the time the paper \cite{Caffarelli-Acta} appeared: this and the next two sections really enter in the main new ideas introduced by Luis. From now on the goal is to study the regularity of the ``free boundary'' $\partial \{u>\varphi\} = \partial \{v>0\}$, but before starting let us observe that, without introducing some more assumptions, we cannot expect any regularity for it. 

Fix two smooth functions $\varphi$ and $g$ and a corresponding solution $u$ of the obstacle problem, hence let $\Omega:=\{u>\varphi\}$. Next let $\lambda$ be any smooth compactly supported bump function inside $\Omega$ with $0\leq \lambda \leq 1$ and define the new obstacle
$\varphi':= \varphi + \lambda (u-\varphi)$, cf. Figure \ref{fig:3}. Obviously $u\geq \varphi'$ and $\varphi'|_{\partial D} < g$, moreover because of the smoothness of $u$ in $\Omega$ and the fact that $\lambda$ is compactly supported in $\Omega$, $\varphi'$ is $C^\infty$ if $\varphi$ is $C^\infty$. We claim that $u$ must be the solution of the obstacle problem with data $g$ on $\partial D$ and obstacle $\varphi'$. This is however obvious: if we were to find a competitor $u'$ which beats the energy of $u$ while staying above $\varphi'$ and taking the boundary data $g$, this competitor would disprove the minimality of $u$ with the obstacle $\varphi$, because $u'\geq \varphi' \geq \varphi$. 

\begin{figure}[htbp]
\begin{center}
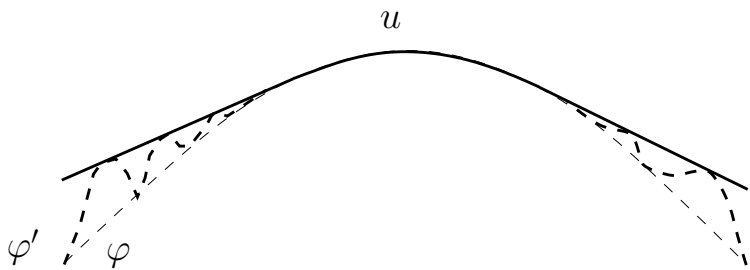
\end{center}\label{f:bad-obstacle}
\caption{A visual explanation of the construction of a bad obstacle, in one space dimension. The thin dashed line is the graph of the obstacle $\varphi$, while the thick line is the solution of some obstacle problem with boundary data $g>\varphi$. The thick dashed line exemplifies the graph of a new obstacle $\varphi'$ which is smooth and lies between $g$ and $\varphi$, while touching $g$ at several points (these additional points form the set $K$ of the text.}\label{fig:3}
\end{figure}

On the other hand, having fixed any compact set $K\subset \Omega$, we can find a $\lambda$ such that $K= \{\lambda=1\}$. So for the obstacle $\varphi'$ the corresponding $\Omega'$ equals $\Omega\setminus K$. As $K$ could be any compact set with nonempty interior, $K$ would, by definition, be part of the free boundary. As $K$ we could however choose even a set with positive Lebesgue measure. 

There is a simple way to exclude such pathological examples. Assume for instance that $K$ were a set of positive measure. Then $u-\varphi'$ would be a nonnegative smooth function which vanishes on $K$, hence its second derivative would also vanish a.e. on $K$. Since $u$ is harmonic, we infer that $\Delta \varphi' = 0$. Such example would then not exist if we assume that the Laplacian of the obstacle is strictly negative. 
While this might seem as an ad hoc assumption, consider that after all strictly concave functions would be a rather large and natural class of obstacles, and the assumption $\Delta \varphi < 0$ is a considerably weaker one.

As a model problem which does not really make a big difference in the analysis, we will in fact assume $\Delta \varphi = -1$\footnote{Assuming $\Delta \varphi$ constant has the great advantage that partial derivatives of $v$ are harmonic functions and will keep our presentation rather simple: in order to handle the case of a general smooth (and positive!) $\Delta \varphi$ we would otherwise need more technical results.}. Under the latter assumption, we note that our function $v$ is a $C^{1,1}$ function which satisfies the conditions
\begin{equation}\label{e:reformulation}
\left\{\begin{array}{l}
\Delta v = \mathbf{1}_{v>0}\\
v\geq 0\\
v|_{\partial D} > 0\, .
\end{array}
\right.
\end{equation}

\subsection{Lower bound on the growth of $v$} As outlined in Section \ref{s:due}, our first task is to show that the function $v$ grows indeed {\em at least} quadratically away from the ``contact set $\{u=\varphi\}=\{v=0\}$'' (analogous estimates are known for the Bernoulli problem and the two-phase problem, but not for the thin obstacle; the availability of such an estimate is often called ``nondegeneracy''). More precisely 

\begin{lemma}\label{l:lower-growth}
Assume $x_0\in \overline{\{v>0\}}$ and $\overline{B}_r (x_0)\subset D$. Then
\begin{equation}\label{e:lower-growth}
\max_{\overline{B}_r (x_0)} v \geq \frac{r^2}{2n}\, .
\end{equation}
\end{lemma} 

Note that by continuity of $v$ it suffices to prove the lower bound when $x_0\in \{v>0\}$. In the latter case consider the function $w (x):= v(x)-\frac{|x-x_0|^2}{2n}$. Because $w (x_0) = v(x_0) > 0$ and because $w$ is negative on the complement of $\{v>0\}$ and harmonic on $\{v>0\}$, $w$ must achieve its positive maximum in $\overline{B}_r (x_0)$ at some point $\bar x$ in $\{v>0\} \cap \partial B_r (x_0)$. But then
\[
0 < w (\bar x) = v (\bar x) - \frac{r^2}{2n}
\]
implying $\max_{\overline{B}_r (x_0)} v \geq v (\bar x) > \frac{r^2}{2n}$. 

\subsection{Blow-ups}\label{s:bu} The existence and nontriviality of blow-ups is now practically an exercise. Let us fix $x_0\in D\cap \partial \Omega$ and without loss of generality, given the translation invariance of the problem, assume that $x_0=0$. As already mentioned in Section \ref{s:tre} the pivotal rescaling is 
\begin{equation}\label{e:rescaling}
v_r (x):= \frac{v (rx)}{r^2}\, .
\end{equation}
Note that $v_r (0)= v (0)=0$ and $\nabla v_r (0)= \nabla v (0)=0$. Moreover, the domain of each such function is $r^{-1} D$: as $r\downarrow 0$ such domains ``invade'' the whole space $\mathbb R^n$.

Recall next that we have Frehse's $C^{1,1}$ estimate. The latter is essentially an estimate on the second derivative: it says that, even though not continuous, it is bounded. The rescaling above preserves however the size of the second derivative, i.e. $D^2 v_r (x) = D^2 v (rx)$. Thus the Liphscitz constant of $\nabla v_r$ is uniformly bounded. Moreover, given that $\nabla v_r (0)=0$ and $v_r (0)=0$ we immediately infer that, on every fixed ball $B_R (0) \subset \mathbb R^n$, both $\|v_r\|_{C^0 (B_R (0))}$ and $\|\nabla v_r\|_{C^0 (B_R (0))}$ enjoy uniform bounds (the first with a quadratic dependence on $R$ and the second with a linear dependence). But then by the Ascoli-Arzel\'a theorem, for every sequence $r_k\downarrow 0$ there is a subsequence, not relabeled, such that $v_{r_k}$ converges uniformly on compact subsets of $\mathbb R^2$ to some function. It is also not difficult to see that the limit $v_0$ is as well $C^{1,1}$ and that $\nabla v_0 (0)=0$ and $v_0 (0)=0$.  $v_0$ is what we will call ``a blow-up'' at the point $x_0\in \partial \Omega$\footnote{For general $x_0\in \partial \Omega$ the rescaled functions will take the form $v_r (x)= r^{-2} v (x_0+rx)$. This essentially amounts to first translate the solution so that $x_0$ is the origin, and then use the formula \eqref{e:rescaling}.}. 

\medskip

Concerning the ``nontriviality'', consider now a blow-up $v_0$. We can introduce the set $\Omega_0 = \{v_0>0\}$. Fix a $\rho>0$ and note that the lower bound 
\[
\max_{\overline{B}_{r\rho}} |v| \geq \frac{(r\rho)^2}{2n} 
\]
translates into
\[
\max_{\overline{B}_\rho} |v_r|\geq \frac{\rho^2}{2n}\, .
\]
Since $v_0$ is the uniform limit of some $v_{r_k}$, we immediately infer 
\[
\max_{\overline{B}_\rho} |v_0|\geq \frac{\rho^2}{2n}\, .
\]
In particular the blow-up $v_0$ is nontrivial, namely $\Omega_0\neq\emptyset$. In fact we just proved that $0\in \partial \Omega_0$. 

\section{Convexity estimate and classification}\label{s:sei}

Fix then a point $x\in \Omega_0$ and use the continuity of $v_0$ to find a ball $\overline{B}_s (x)$ over which the minimum of $v_0$ is positive. It turns out that for $k$ large enough also the minimum of $v_{r_k}$ will be positive and hence $v_{r_k} (\cdot) = r_k^{-2} v (r_k \cdot)$ is smooth on the ball. In fact, since $\Delta v_{r_k} (y) = \Delta v (r_k y)$, $\Delta v_{r_k} \equiv 1$ on $B_s (x)$. But then classical estimates for harmonic functions (given that $y\mapsto v_{r_k} (y)-\frac{|y|^2}{2n}$ is harmonic!) imply that $v_{r_k}$ converges smoothly to $v_0$ on $B_s (x)$ (and thus on any compact subset of $\Omega_0$). So $\Delta v_0 = 1$ on $\Omega_0$: our blow-up ``inherited'' the same PDE of the rescaled functions. 

\medskip

It would be wonderful if we had a Liouville-type theorem classifying such solutions. Note that we even have an upper bound, since Frehse's estimate, by the very computations which we used to gain the lower bound, proves that $|v_0 (x)|\leq C |x|^2$, $|\nabla v_0 (x)|\leq C |x|$, and $|D^2 v_0 (x)|\leq C$ for some positive constant $C$ (there are points where $\nabla v_0$ is not differentiable, however the estimate in the second derivative is correct whenever the latter exists). A complete classification has in fact been achieved: first in \cite{Sakai} for $n=2$, but only very recently for $n\geq 6$, \cite{ESW}, and finally for the remaining cases $3\leq n \leq 5$ in \cite{EFW}. An important tool is indeed the convexity estimate which is a key idea introduced in \cite{Caffarelli-Acta}: a global solution is in fact convex. However, rather than proving the latter result, we will argue for an ``almost convexity estimate'' on the function $v$ ``prior to the blow-up'' and hence infer 

\begin{lemma}\label{l:convexity}
Every blow-up $v_0$ is a convex function. 
\end{lemma}

The proof will be given in Section \ref{s:almost-convexity}, where we will show the existence of some nonnegative function $r\mapsto \omega (r)$ which converges to $0$ as $r\downarrow 0$ with the property that, for every unit vector $e\in \mathbb R^2$, $D^2_{ee} v\geq - \omega (r)$ on $B_r (0)$.  Before coming to that let us however explore some consequences of Lemma \ref{l:convexity}

\subsection{Subdivision of the free boundary $\partial \Omega$} Assuming Lemma \ref{l:convexity} we now can conclude that $\{v_0=0\}= \mathbb R^n \setminus \Omega_0$ is a convex set. We further distinguish two cases:
\begin{itemize}
\item The blow-up is degenerate, namely the convex set $\{v_0=0\}$ has no interior points. Given that it is convex and contains the origin, it is necessarily contained in an hyperplane $H_0$.
\item The blow-up is nondegenerate, namely the convex set $\{v_0=0\}$ has nonempty interior. 
\end{itemize}
We then introduce the subdivision of the free boundary in the following two pieces.

\begin{definition}\label{d:free-boundary-sub}
A point of the free boundary $D\cap \partial \Omega$ is said ``regular'' if {\em at least one} blow-up is nondegenerate.
Any point of the free boundary which is not regular will be called ``singular'' (in particular all blow-ups at a singular points are necessarily degenerate). 
\end{definition}

The terminology will be later vindicated by the following $\varepsilon$-regularity theorem, already mentioned in Section \ref{s:tre} and which will be proved in Section \ref{s:otto}.

\begin{theorem}\label{t:reg}
At any regular point $x_0\in D\cap \partial \Omega$ there is a neighborhood where $\partial \Omega$ is a $C^1$ hypersurface. 
\end{theorem}

In fact the argument will show more, namely that the tangent to the free boundary is also H\"older continuous (in the PDE jargon this is the $C^{1,\alpha}$ regularity already explained). Note in passing that the $\varepsilon$-regularity will in fact imply that:
\begin{itemize}
\item The blow-up at $x_0$ is unique (the rescaled functions $v_r$ converge all to the same limit as $r\downarrow 0$);
\item The latter is in fact a one-variable function, the very ``model'' discovered in Section \ref{s:tre}. 
\end{itemize}

\subsection{Degenerate blow-ups}
Before getting to the convexity estimate and to the $\varepsilon$ regularity theory, let us however analyze what the possible shape of a degenerate blow-up $v_0$ might be. Summarizing the information gained so far, we know that, if $v_0$ us a degenerate blow-up, then there is an hyperplane $H_0$ with the following properties:
\begin{itemize}
\item[(a)] $v_0$ is $C^1$ and $C^2$ on each side of $H_0$, with bounded second derivatives;
\item[(b)] $\Delta v_0=1$ on each side of $H_0$;
\item[(c)] $v_0(0)=0$ and $\nabla v_0 (0)=0$.
\end{itemize}
Consider now the function $x\mapsto w_0 (x)= v_0 (x)-\frac{|x|^2}{2n}$: $w_0$ is smooth on each of the two open half-spaces delimited by $H_0$ (denote them by $H^+$ abd $H^-$ and it is $C^1$ over $H_0$. It is then easy to show
shows that $w$ is stationary for the Dirichlet integral on every bounded $U\subset \mathbb R^n$: we can integrate by parts the expressions $\int_{U\cap H^\pm} \nabla w_0 \cdot \nabla \psi$ for every smooth $\psi\in C^\infty_c (U)$ and the boundary terms on $U\cap H_0$ cancel out when we sum the them. In particular, by the Weyl Lemma, $w_0$ is harmonic and smooth. Observe also that its second derivatives are bounded. So the classical Liouville theorem implies that the second derivatives are constant. In particular also the second derivatives of $v_0$ are constant. Given (c) we conclude that $v_0$ is an homogeneous quadratic polynomial.

\medskip 

We just make one final remark. At a singular point $x_0$ the set $\{v=0\} \cap B_r (x_0)$ becomes ``thinner and thinner'' as $r\downarrow 0$, in fact $|\{v=0\}\cap B_r (x_0)| = o (r^{n-1})$. Pick indeed any sequence $r_k\downarrow 0$. A subsequence, not relabeled, of $v_{r_k}$ must converge uniformly to a degenerate blow up $v_0$. The latter is however positive outside some hyperplane $H_0$. In particular, for any fixed $\delta>0$, provided $k$ is sufficiently large, $\{v_{r_k}=0\}\cap B_1$ must be contained in a $\delta$ neighborhood of $H_0$. Scaling back the information to the function $v$, this translates into $\{v=0\}\cap B_{r_k} (x_0)$ being contained in a $\delta r_k$-neighborhood of $(x_0+H_0)\cap B_{r_k} (x_0)$, in turn implying that $|\{v=0\}\cap B_{r_k} (x_0)|\leq C \delta r_k^{n-1}$ for some geometric constant $C$ (which in particular is independent of $\delta$). The arbitrariness of $\delta$ shows therefore that 
\[
\lim_{r\downarrow 0} \frac{|\{v=0\}\cap B_r (x_0)|}{r^{n-1}} = 0\, .
\] 

\subsection{Proof of the convexity estimate}\label{s:almost-convexity} We now wish to sketch a proof of Lemma \ref{l:convexity}. In order to avoid technicalities we will always assume that we are dealing with points ``well inside $D$'', in such a way that their closest point in $\partial \Omega$ is also in the interior of $D$.

We start with the following very elementary observation.

\begin{lemma}\label{l:elem}
Assume that $v (x_0)>0$ and that $\partial B_r (x_0)$ touches $\partial \Omega$ at a point $y\in D$, while $B_r (x_0)\subset \Omega$. Then for every unit vector $e$ and for any $\delta< \frac{1}{2}$ there is a point $\bar x_1$ in $B_{(1-\delta/2) r} (x_0)$ at which we have the estimate
\[
D_{ee} v (\bar x_1) \geq - C \sqrt{\delta}
\]  
where the constant $C$ depends only on the estimate in Frehse's $C^{1,1}$ theorem. 
\end{lemma}
\begin{proof}
By translation and rotation we can assume $y=0$. Now for an appropriate geometric constant $c_0>0$ we are guaranteed that $B_{(1-\delta/2) r} (x_0)$ contains at least one of the two segments $\sigma_1 = [\delta x_0, \delta x_0 + c_0 \sqrt{\delta} r e]$ and $\sigma_2 = [\delta x_0, \delta x_0 - c_0 \sqrt{\delta} r e]$. In fact the ``worst case'' is when $e$ is perpendicular to the vector $x_0$, cf. Figure \ref{fig:2}.

\begin{figure}[htbp]
\begin{center}
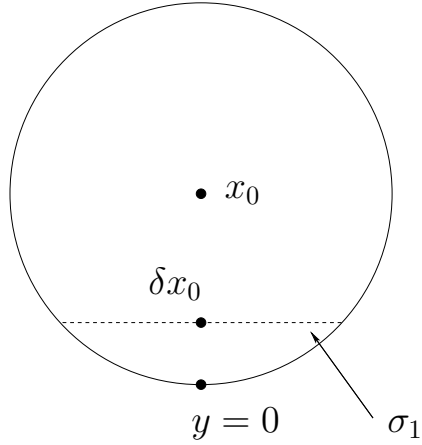
\end{center}
\caption{The segment $\delta_1$ when $e$ is orthogonal to $x_0$.}\label{fig:2}
\end{figure}

Up to change of sign of $e$ (which does not affect the statement of the lemma, since the second partial derivative $D_{ee} v$ remains the same!) assume we are in the first case and set $x_1= \delta x_0 + c_0 \sqrt{\delta} r e$. Applying the fundamental theorem of calculus twice on the segment it is easy to see that 
\[
0\leq v (x_1) \leq v (\delta x_0) + c_0 \sqrt{\delta} r |D_e v (\delta x_0)| + \tfrac{1}{2} c_0^2 \delta r^2 \sup_{\sigma_1} D_{ee} v\, . 
\]  
Consider that, since $v(0)=0$, $\nabla v (0)=0$ and $\sup |D^2v| \leq C_0$ by Frehse's theorem, 
\begin{align*}
v (\delta x_0) &\leq C_0 (\delta r)^2/2 \\
|D_e v (\delta x_0)| &\leq C_0 \delta r\, .
\end{align*}
We therefore conclude that 
\[
 \tfrac{1}{2} c_0^2 \delta r^2 \sup_{\sigma_1} D_{ee} v \geq - \tfrac{1}{2} C_0 (\delta r)^2 - C_0 \delta^{\frac{3}{2}} r^2\, .
\]
Since $\sigma_1 \subset B_{(1-\delta/2) r} (x_0)$ the claim of the lemma readily follows.
\end{proof}

Fix now a vector $e$ with norm $1$. 
We wish to use the above lemma and the Harnack estimate to prove the following statement 
\begin{itemize}
\item[(I)] Assume that for some $r>0$ and for some number $M>0$ we have the estimate $D_{ee} v (x) > - M$ for every $x$ at distance at most $r$ from $D\cap \partial \Omega$ and sufficiently far away from $\partial D$. Then for every $x$ at distance at most $r/2$ from $D\cap \partial \Omega$ and still sufficiently far away from $\partial D$ we have the improved inequality
\[
D_{ee} v (x) \geq -M + c_1 M^{2n-1} = - M \big(1-c_1 M^{2n-2}\big)\, ,
\] 
where $c_1$ depends only on geometric constants and the estimate of Frehse's theorem. 
\end{itemize}
I will not attempt at specifying what exactly ``sufficiently far away from $\partial D$'' means, the interested reader will be able to fill in the details after I present the argument. Before coming to the proof of the claim (I), we finish the proof of the convexity estimate. Recall that Frehse's theorem gives us some uniform bound on the absolute value of $D^2 v$, so we do know that $D_{ee} v > - M_0$ for some $M_0$. We can then pick some $r_0>0$ and keep ``far away from'' $\partial D$ to infer 
\[
D_{ee} v (x) \geq -M_0 + c_1 M_0^{2n-1} = - M_0 (1-c_1 M_0^{2(n-1)}) =: - f (M_0)
\]
when $x$ is at distance at most $\frac{r}{2}$ from $D\cap \partial \Omega$. 
If $f (M_0)$ were negative or $0$ we would be rather happy (because we would immediately conquer $D_{ee} v \geq 0$ in a neighborhood of $D\cap \partial \Omega$ and away from $\partial D$). Otherwise we can keep iterating, i.e. at distance $2^{-k} r_0$ we will have $-f^k (M_0)$ as lower bound, where $f^k = f \circ \ldots \circ f$ is the $k$-th iterate of the map $f$. Observe that, because we are assuming $f (M_0)>0$, we have immediately $0<f(M_0)<M_0$, hence also $0< f^2 (M_0) < f(M_0)$, and so on: $f^k (M_0)$ stays always positive. It is now pretty simple to argue that the monotone sequence $M_k = f^k (M_0)$ must in fact converge to $0$, which means that our lower bound on $D_{ee} v$ will approach zero as we get closer and closer to $D\cap \partial \Omega$ (but staying away from $\partial D$!). Indeed if it were $M_\infty = \inf_k M_k = \lim_k M_k >0$, then setting 
\[
\gamma := (1-c_1 M_\infty^{2(n-1)})
\]
we would immediately conclude the recursive estimate $M_k \leq \gamma M_{k-1} \leq \ldots \leq \gamma^k M_0$ and hence infer that $M_k$ not only converges to $0$, but it converges geometrically fast to it (in fact the actual converge of $M_k$ is pretty slow). 

It remains to show (I): the deus ex machina is the celebrated Harnack's estimate, one of the workhorse of the theory of elliptic PDEs in the last century or so. For harmonic functions it amounts to the following statement, proved by Harnack in the nineteenth century.

\begin{theorem}\label{t:Harnack}
Assume $f$ is a nonnegative harmonic function on $B_R (x_0)\subset \mathbb R^n$ and let $\bar x_1$ be a point in $B_r (x_0)$ with $r<R$. Then 
\[
f (\bar x_1) \frac{1+(\tfrac{r}{R})^{n-1}}{1-\tfrac{r}{R}} \geq f(x_0) \geq \frac{1 - (\tfrac{r}{R})^{n-1}}{1+\tfrac{r}{R}} f (\bar x_1)\, .
\]
\end{theorem}

Now, to the proof of (I). Fix $r$ and $M$ as in the statement and let $x_0$ be a point at distance at most $\frac{r}{2}$ from $\partial \Omega$. In fact let $r'$ be the largest radius centered at $x_0$ over which $v$ is positive, so that $\partial B_{r'} (x_0)$ touches $D\cap \partial \Omega$. Observe that $D_{ee} v$ is an harmonic function in $B_{r'} (x_0)$ (since $\Delta v \equiv 1$ in it, any partial derivative of $v$ is harmonic!). Moreover $D_{ee} v+M$ is positive by assumption (because $B_{r'} (x_0)$ is contained in the $r$-neighborhood of $\partial \Omega$). We apply now the Harnack inequality between $x_0$ and the point $\bar x_1$ found in Lemma \ref{l:elem}, where for the moment $\delta$ is not yet chosen. Then we find
\[
D_{ee} v (x_0) \geq -M + \tfrac{\delta^{n-1}}{2^{n-1}} (M - C\sqrt{\delta})\, .
\] 
We now specify $\delta$ to be smaller than $\frac{M^2}{4C^2}$ and, of course, smaller than $1$ (i.e. we set it equal to the minimum of the two) and the desired conclusion follows. 

\section{The $\varepsilon$-regularity theorem}\label{s:sette}

Our task is now to prove Theorem \ref{t:reg}. We will in fact first sketch a proof of the regularity of the free boundary for a nondegenerate blow-up. Consequently we will leverage the latter to transfer the same amount of regularity to a rescaling which is sufficiently close to it. 

We fix therefore a nondegenerate blow-up $v_0$. Recall that that on $v_0: \mathbb R^n \to [0, \infty)$ we have therefore the following pieces of information.
\begin{itemize}
\item[(a)] $v_0$ is convex, $C^1$, and enjoys a uniform upper bound on the second derivative;
\item[(b)] $\Delta v_0=1$ on $\{v_0>0\}$;
\item[(c)] $v_0 (0)=0$ and $\nabla v_0 (0)=0$;
\item[(d)] $\max \{v_0 (x): |x|\leq r\} \geq (2n)^{-1} r^2$;
\item[(e)] the convex set $\{v_0=0\}$ has nonempty interior.
\end{itemize}
Note that just (e) and classical properties of convex sets guarantee the existence of a radius $r_0>0$ with the property that, up to a change of coordinates, there is a domain $U\subset \mathbb R^{n-1}$ and a concave function $f: U \to \mathbb R$ such that 
\[
\{v_0=0\}\cap B_{r_0} (0) = \{(x', x_n)\in B_{r_0} (0): x'\in U \;\mbox{and}\; x_n\leq f (x')\}\, .
\] 
The goal of this section is to improve the statement to continuous differentiability of the function $f$. In fact the proof will also give a H\"older modulus of continuity for the derivative. The following is the explicit statement.

\begin{proposition}\label{p:first-regularity}
For a sufficiently small $r_0>0$ the set $\partial \{v_0>0\}\cap B_{r_0} (0)$ is a $C^{1,\alpha}$ hypersurface.
\end{proposition}

Before coming to the proof of Proposition \ref{p:first-regularity} we pause for a moment to observe a quite interesting corollary. Obviously, $\partial \{v_0>0\}$ has a tangent plane $H_0$ at the point $0$. Assume w.lo.g. that $H_0=\{x_n=0\}$ and that the convex set $\{v_0=0\}$ lies below $H_0$. We can now further ``blow-up'' $v_0$ at the origin, namely look at the uniform limits of the rescalings
\[
v_{0,r} (x) := \frac{v_0 (rx)}{r^2}\, .
\]
All the considerations which applied to the blow-ups of $v$ apply now to any blow-up $v_{0,0}$ of $v_0$: the latter also satisfies all the properties (a)-(e) listed above. Moreover the uniform convergence of the rescalings and the regularity just gained for the free boundary of $\{v_0=0\}$ imply that $\{v_{0,0}=0\}$ is in fact the half-space $\{x_n \leq 0\}$. Consider now on the upper half-space the harmonic function
\[
h  (x) := v_{0,0} (x) - \frac{x_n^2}{2}\, .
\]
The latter is identically $0$ on $H_0$ and by a Schwartz reflection we can extend it harmonically to all of $\mathbb R^n$. The quadratic lower bound and Liouville's theorem imply that $h$ is a second order polynomial, but in fact it has to be homogenous and quadratic because $h_{0,0} (0)=0$ and $\nabla h_{0,0} (0)=0$. However a quadratic polynomial which vanishes on the hyperplane $H_0$ is necessarily a function of the variable $x_n$ only and, because of the harmonicity, this forces $h$ to vanish identically. So there is a unique blow-up of $v_0$ at the origin and it is the function
\begin{equation}\label{e:v00}
v_{0,0} (x) = 
\left\{
\begin{array}{ll}
{\textstyle{\frac{x_n^2}{2}}} \qquad & \mbox{if $x_n>0$}\\
0 & \mbox{otherwise}
\end{array}
\right.
\end{equation}
Now, $v_0$ is the limit of some sequence $v_{r_k}$ with $r_k\downarrow 0$, while we just inferred that $v_{0,0}$ is the limit of $v_{0,s}$ for $s\downarrow 0$. It is a simple routine exercise to extract a diagonal sequence 
$r_k s_k$ with $s_k\downarrow 0$ such that $v_{r_k s_k}\to v_{0,0}$. In other words we just inferred that at a regular point at least one blow-up is the ``expected model'' of Section \ref{s:tre}. We pause for a second to state the corollary we just found.

\begin{corollary}\label{c:one-good-bu}
If $x_0\in D\cap \partial \{v>0\}$ is a regular point, then, up to applying a rotation, at least one blow-up at $0$ is the function $v_{0,0}$ of \eqref{e:v00}. 
\end{corollary}

Roughly two thirds of this section will be devoted to the proof of Proposition \ref{p:first-regularity}. The remaining one will be dedicated to the very last step of the proof that at a regular point of $v$ the free boundary is indeed regular. Considering Corollary  \ref{c:one-good-bu}, this last step can in fact be reduced to the following $\varepsilon$-regularity statement.

\begin{theorem}\label{t:real-eps-reg}
There is a $\varepsilon>0$ with the following property: if $0\in \partial \{v>0\}$, $B_2\subset D$, and $\|v-v_{0,0}\|_{C^1 (B_2)}$ is smaller than $\varepsilon$, then $B_1\cap \partial \{v>0\}$ is a $C^{1,\alpha}$ hypersurface.  
\end{theorem} 

\subsection{Lipschitz regularity for $\partial \{v_0>0\}$}\label{s:Lipschitz} In the first step we will derive the Lipschitz regularity of $\partial \{v_0>0\}$ in a neighborhood of $0$ using a ``PDE argument''. Even though we already know this from the convexity of $\{v_0=0\}$, the argument will in fact be useful to prove Theorem \ref{t:real-eps-reg} because we will be able to export it to the ``perturbed'' setting of the theorem.

\medskip 

Recall that by assumption $\{v_0=0\}$ contains a ball $B_{2s} (z)$. By possibly choosing $s$ even smaller, we will assume that $B_s (0)$ and $B_{2s} (z)$ are disjoint. Consider first the set of directions 
\[
W:=\{-\textstyle{\frac{y}{|y|}} : y\in B_s (z)\}\, .
\] 
We claim that 
\begin{equation}\label{e:positivity}
\frac{\partial v_0}{\partial y} (x) \geq 0 \qquad\forall y\in W, \,  \forall x\in B_s (0)\, .
\end{equation}
In fact for every $x\in B_s (0)$ and $y\in W$ consider the point $z':= x+w$ where $w\in B_s (0)$ is such that $y= -\frac{w}{|w|}$. 
Observe that 
\[
\frac{d}{dt} \left[\frac{\partial v_0}{\partial y} (z'+ty)\right] \geq 0
\]
by convexity. On the other hand $\frac{\partial v_0}{\partial y} (z')=0$ because $z'\in B_{2s} (z) \subset \{v_0=0\}$. Thus \eqref{e:positivity} follows from integrating the latter inequality for $t\in [0,|w|]$. So the harmonic function $\frac{\partial v_0}{\partial y}$ is nonnegative on $B_s (0) \setminus \{v_0>0\}$, while it vanishes on $\partial \{v_0=0\}$: in particular it either vanishes identically on  $B_s (0) \setminus \{v_0>0\}$ or it is strictly positive by the strong maximum principle. But it is not difficult to see that it cannot vanish identically: if $\frac{\partial v_0}{\partial y}$ were to vanish everywhere, then $v_0$ would vanish on the set $B_{2s} (z) + \mathbb R y$ simply by integration along the lines $z'+\mathbb R y$ for $z'\in B_{2s} (z)$, but the latter set contains a neighborhood of the origin, which instead belongs to $\partial \{v_0>0\}$. 

Now, the function $\frac{\nabla v_0}{|\nabla v_0|}$ is the unit normal to the level sets of $v_0$ on all points where it does not vanish. The above estimate shows that 
\[
\frac{\nabla v_0}{|\nabla v_0|} \cdot y >0 \qquad \forall y \in W 
\]
on $\{v_0>0\}\cap B_s (0)$. This not only implies that the vector field $\frac{\nabla v_0}{|\nabla v_0|}$ does not vanish in $\{v_0>0\}\cap B_s (0)$, but it also implies that it belongs to an appropriate cone of directions. Geometrically this means that the level sets $\{v_0>0\}$ are graphs of Lipschitz functions over the plane $\pi$ orthogonal to $z$, with a Lipschitz constant which depends only on $s$, cf. Figure \ref{fig:4}. In particular this constant is under control if we consider level sets of the form $\{v_0=\varepsilon\}$ for small $\varepsilon>0$: as we let $\varepsilon$ approach $0$, we infer the same Lipschitz regularity for $B_s (0) \cap \partial \{v_0>0\}$. 

\begin{figure}[htbp]
\begin{center}
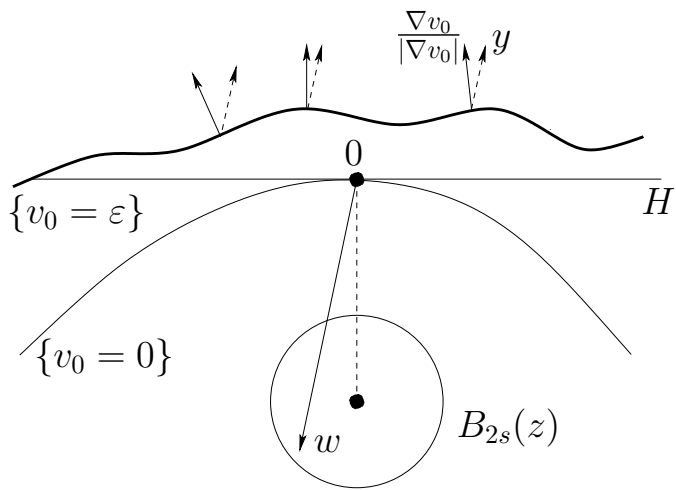
\end{center}
\caption{An illustration of the Lipschitz bound for the level set $\{v_0=\varepsilon\}$, which in the picture is represented by the thick line. $w$ is an element of $B_{2s} (z)$, while $y=-\frac{w}{|w|}$ is the corresponding element in $W$: dashed arrows coincide with it. The normals to the level set are given by the values of $\frac{\nabla v_0}{|\nabla v_0|}$: since their scalar products with all the elements $y\in W$ must be positive, these normals cannot belong to the ``horizontal plane'' $H$ orthogonal to $z$, in fact they form an angle with $H$ which is larger than a constant $c$ which depends on $s$.}\label{fig:4}
\end{figure}

\subsection{$C^{1,\alpha}$ regularity of $\partial \{v_0>0\}$} This is the most technical part of the proof. If we fix any two vectors $y_1, y_2$ in $W$, we can take the ratio
\[
\frac{{\textstyle{\frac{\partial v_0}{\partial y_1}}}}{{\textstyle{\frac{\partial v_0}{\partial y_2}}}}
\]
and, because $\{v_0>0\}\cap B_s (0)$ is a Lipschitz domain and both functions vanish on $\{v_0=0\}$, a hard PDE result (the boundary Harnack inequality) implies that the ratio is in fact H\"older continuous, for some fixed H\"older exponent. In fact this comes with a bound on the H\"older seminorm\footnote{The H\"older seminorm of a function $f$ with exponent $\alpha$ is $[f]_\alpha = \sup_{x\neq y} \frac{|f(x)-f(y)|}{|x-y|^\alpha}$.} in inner balls (for instance in $B_{s/2} (0) \cap \{v_0>0\}$) which depends only on the Lipschitz constant of $\partial \{v_0>0\}\cap B_s (0)$. 

Even for several PDE experts, the fact that this estimate works under the low regularity assumption of Lipschitz boundaries will probably look like a nontrivial fact: for a proof they can consult \cite[Theorem 11.6]{Caffarelli-Salsa}. 
For the nonexpert I wish to give at least some rough explanation. Assume for the moment that the boundary $\partial \{v_0>0\} \cap B_s (0)$ is flat and to fix ideas let us take $s=1$. The domain will then be denoted by $B_1^+$, the upper half ball $B_1 (0) \cap \{x_n>0\}$. The claim is that, if we take two positive harmonic functions on $B_1^+$, say $\gamma$ and $\beta$, which vanish on the flat part of the boundary, namely $B_1 (0) \cap \{x_n=0\}$, then the ratio $\frac{\gamma}{\beta}$ is in fact smooth up to the boundary. 

We can extend both functions by Schwartz reflection to $B_1$ and we keep denoting such extensions by $\gamma$ and $\beta$. One of the two functions could well be $x_n$ and in fact to prove our claim it will suffice to just show that $\frac{\gamma}{x_n}$ and $\frac{x_n}{\gamma}$ are both regular. Consider $\frac{\gamma}{x_n}$. The latter is obviously a well defined function on $B_s\setminus \{x_n=0\}$, which extends continuously on $x_n=0$ to the value of $\frac{\partial \gamma}{\partial x_n}$. The same consideration applies to $\frac{\partial}{\partial x_i} \left(\frac{\gamma}{x_n}\right)$ for $i=1, \ldots, n-1$, as the partial derivatives $\frac{\partial \gamma}{\partial x_i}$ all vanish on $\{x_n=0\}\cap \partial B_1$. Differentiating in $x_n$ yields the formula 
\[
\frac{\partial}{\partial x_n} \left(\frac{\gamma}{x_n}\right) = \frac{1}{x_n^2} \left(x_n \frac{\partial \gamma}{\partial x_n} - \gamma\right)\, .
\]
We claim that the function extends continuously to $0$ on $\{x_n=0\}$: the reader can use the Taylor's formula in the $x_n$ and note that, given $\frac{\partial^2 \gamma}{\partial x_n^2} = -\sum_{i\leq n-1} \frac{\partial^2 \gamma}{\partial x_i^2}$, $\frac{\partial^2 \gamma}{\partial x_n^2}$ must vanish on $\{x_n=0\}$ because all the tangential derivatives vanish on $\{x_n=0\}$. We can now iterate these considerations, using the harmonicity of the partial derivatives of $\gamma$: it is not difficult to see that $\frac{\gamma}{x_n}$ is then $C^2$ (in fact the function is real analytic, but this is besides our point). 

What about $\frac{x_n}{\gamma}$? It would suffice to show that $\frac{\gamma}{x_n}$ is bounded way from zero over $B_{1/2} (0)$ and then the desired regularity would follow from the formula $\frac{x_n}{\gamma} = \left(\frac{\gamma}{x_n}\right)^{-1}$. However, given that $\frac{\partial \gamma}{\partial x_n}$ is positive on $\{x_n=0\}$ by Hopf's Lemma (which in fact is just a manifestation of the strong maximum principle), it is obvious that, for every constant $C$ sufficiently large, the inequality $C\gamma \geq x_n$ in a neighborhood of $B_{1/2} (0) \cap \{x_n=0\}$ in the upper half space, i.e. in $B_{1/2} (0) \cap \{0\leq x_n\leq \varepsilon\}$ for some positive $\varepsilon$. But because $\gamma$ is bounded away from zero on $B_{3/4} (0) \cap \{x_n\geq \varepsilon\}$ while $x_n$ is bounded from above on the same domain, the same inequality $x_n \leq C \gamma$ holds there at the price of maybe taking a larger constant $C$. Summarizing $C \gamma \geq x_n$ on the upper half ball whenever $C$ is large enough. So $\frac{\gamma}{x_n}\geq C^{-1}$ on the same domain. But then because both functions $\gamma$ and $x_n$ are odd in $x_n$, the same estimate holds on the lower half ball. This gives the desired control from below.  

We now resume our discussion. We have concluded that 
\[
\frac{{\textstyle{\frac{\partial v_0}{\partial y_1}}}}{{\textstyle{\frac{\partial v_0}{\partial y_2}}}}
\]
is H\"older for every choice of $y_1, y_2\in W$. Because $W$ spans $\mathbb R^n$, we can choose to vary $y_1$ in a basis and conclude that  
\[
\frac{\nabla v_0}{{\textstyle{\frac{\partial v_0}{\partial y_2}}}}
\]
is H\"older continuous. However, the modulus of an H\"older continuous vector field is also H\"older continuous and thus we conclude that 
\begin{equation}\label{e:ratio}
\left({\textstyle{\frac{\partial v_0}{\partial y_2}}}\right)^{-1} |\nabla v_0|
\end{equation}
is H\"older continuous. Recall that indeed $\frac{\partial v_0}{\partial y_2}$ is positive. In fact the ratio is bounded away from zero (again by the boundary Harnack inequality). So the inverse of the ratio in \eqref{e:ratio} is as well H\"older and, because the product of H\"older functions is H\"older, we finally conclude that $\frac{\nabla v_0}{|\nabla v_0|}$ is H\"older.

We now can invoke the same argument of the previous section: the vector field $\frac{\nabla v_0}{|\nabla v_0|}$ gives in fact the unit normal to the level sets of $v_0$ on $\{v_0>0\}$ and its H\"older continuity means the H\"older continuity of the tangent plane to such level sets. Since our estimate is independent of which level set we are looking at, it can be transferred to $\partial \{v_0=0\}$ (as we transferred the Lipschitz bound). 

\subsection{Proof of Theorem \ref{t:real-eps-reg}} The final step is to transfer the conclusions of the previous two sections to the ``approximate'' situation of Theorem \ref{t:real-eps-reg}. First of all, observe that $\{v=0\}$ is in fact ``close'' to $\{v_{0,0}=0\}$ in the sense that, for every fixed $x_0$ in the interior of $\{v_{0,0}=0\}$, e.g. $x_0 = (0, 0, \ldots , 1)$, is in fact in the interior of $v$ provided $\varepsilon$ is sufficiently small. In fact, assume the contrary: then there would be a sequence $\varepsilon_k\downarrow 0$, a sequence of $v_k$'s satisfying the assumptions of Theorem \ref{t:real-eps-reg} and a sequence of points $x_k$ with $v_k (x_k)\neq 0$ converging to $x_0$. Recall however the lower bound on the growth of $v_k$ by Lemma \ref{l:lower-growth}: 
\[
\max_{\overline{B}_{1/2} (x_k)} v_k \geq  \frac{1}{2n} \frac{1}{2^n}\, .
\]
In particular we conclude
\[
\max_{\overline{B}_{1/2} (x_0)} v_{0,0} \geq \frac{1}{2n} \frac{1}{2^n}\, ,
\]
contradicting $B_{1/2} (x_0) \subset \{v_{0,0} =0\}$. In fact, the above argument, suitably modified, proves the following

\begin{corollary}\label{c:Hausdorff-convergence}
For every radius $R>0$ and every $\delta>0$, if $\varepsilon>0$ in Theorem \ref{t:real-eps-reg} is sufficiently small, then the Hausdorff distance between $\{v=0\}\cap \overline{B}_R (0)$ and $\{v_{0,0}=0\}\cap \overline{B}_R (0)$ is smaller than $\delta$.
\end{corollary}

Next consider $W:= \{- \frac{y}{|y|}:y\in B_{1/4} (x_0)\}$ and $x\in B_{1/4} (0)\cap \{v>0\}$. If we were able to argue as in Section \ref{s:Lipschitz} to conclude $\frac{\partial v}{\partial y} (x) >0$ the rest of the considerations in that section would carry over to show that $\partial \{v>0\}\cap B_{1/8} (0)$ is Lipschitz. In particular from there we would also conclude the $C^{1,\alpha}$ regularity following precisely the argument of the previous section. However, we lack the convexity of $v$, which was the deus ex-machina of the argument in Section \ref{s:Lipschitz}.  

\medskip

We will now remedy to the absence of the latter property and show that indeed $\frac{\partial v}{\partial y} >0$ on $\{v>0\}\cap B_{1/2} (0)$, provided $\varepsilon$ is small enough. Fix in fact a $y$ and fix for the moment a number $\sigma>0$, which we will choose only at the very end. On the domain $B_1 (0) \cap \{\dist (\cdot, \{v_{0,0}=0\})>\frac{\sigma}{2}\}$ the infimum of $\frac{\partial v_{0,0}}{\partial y}$ is some positive number $\lambda$. In particular, if $\varepsilon>0$ is sufficiently small, Corollary \ref{c:Hausdorff-convergence} implies that, on the domain $B_1 (0) \cap \{\dist (\cdot, \{v=0\})\geq \sigma\})\}$ the infimum of $\frac{\partial v}{\partial y}$ must be at least $\frac{\lambda}{2}$. If we consider the harmonic function 
\[
h := \frac{2}{\lambda} \frac{\partial v}{\partial y}
\] 
on $\Omega:= \{v>0\}$, we then conclude that 
\begin{itemize}
\item[(a)] If $x\in \Omega \cap B_1 (0)$ and $\dist (x, \partial \Omega) \geq \sigma$, then $h\geq 1$.
\end{itemize}
On the other hand, because $\frac{\partial v}{\partial y}$ is $\varepsilon$ close to $\frac{\partial v_{0,0}}{\partial y}$, which is nonnegative, again if $\varepsilon>0$ is sufficiently small we have
\begin{itemize}
\item[(b)] $\inf_\Omega h > -\sigma$.
\end{itemize}
Finally, because $\nabla v =0$ on $B_1 \cap \partial \Omega$, 
\begin{itemize}
\item[(c)] $h=0$ on $B_1\cap \partial \Omega$.
\end{itemize}
We will now see that these three conditions imply that in fact $h\geq 0$ on $\Omega\cap B_{1/2}$, if $\sigma$ is smaller than a geometric constant. The simple argument is taken from \cite[Lemma 11]{Caffarelli-JFA} and goes as follows. Assume there is a point $x_0$ at which $h(x_0)<0$. At this point we necessarily have $\dist (x_0, \partial \Omega)<\sigma$, because of (a). Consider then the harmonic function
\[
k (x) = h(x) - \delta \left[v (x)-\frac{1}{2n} |x-x_0|^2\right]\, ,
\]
where $\delta$ is again a nonnegative number which will be chosen appropriately in a few lines. Now $k(x_0)= h(x_0) - \delta v (x_0) <0$. By the maximum principle $v$ must have a negative minimum on $\partial (B_{1/4} (x_0)\cap \Omega)$. The latter minimum cannot be on $\partial \Omega$, because there both $h$ and $v$ vanish and hence $k$ must be positive. Thus
\begin{itemize}
\item[(C)] the infimum of $k$ on $\Omega \cap\partial B_{1/4} (x_0)$ is negative 
\end{itemize}

Next, because $\dist (x_0, \partial \Omega) \leq \sigma$, we also have $v (x) \leq C \sigma^2$ on $\Omega \cap \partial B_{1/4} (x_0) \cap \{\dist (\cdot, \partial \Omega) < \sigma\}$, where $C$ is the constant in Frehse's theorem.  Thus, on that region we have
\[
k \geq -\sigma - C \delta \sigma^2 + \frac{\delta}{2n} \left(\frac{1}{4}\right)^2\, .
\]
We now can declare that
\begin{equation}\label{e:choice-delta}
\delta = 64n \sigma\, ,
\end{equation}
so that we get
\begin{equation}\label{e:choice-sigma-1}
k \geq \sigma - 64C n \sigma^3  \qquad \mbox{on $\Omega \cap \partial B_{1/4} (x_0) \cap \{\dist (\cdot, \partial \Omega)< \sigma$\}}\, .
\end{equation}
On the other hand on $\partial B_{1/4} (x_0) \cap \{\dist (\cdot, \partial \Omega)\geq \sigma\}$ we know that $h\geq 1$, while $v \leq C$, again by Frehse's Theorem. So we infer 
\begin{equation}\label{e:choice-sigma-2}
k \geq 1 - 64 C n \sigma \qquad \mbox{on $\partial B_{1/4} (x_0) \cap \{\dist (\cdot, \partial \Omega)\geq \sigma$\}}\, .
\end{equation}
Since $C$ is a fixed constant (decided by the proof of Frehse's Theorem), clearly a choice of a $\sigma$ suitably small implies through \eqref{e:choice-sigma-1} and \eqref{e:choice-sigma-2} that $k$ is positive on $\Omega \cap \partial B_{1/4} (x_0)$. This is however at odds with statement (C). The contradiction stems from having assumed that $h$ takes a negative value on $B_{1/2}\cap \Omega$. Hence it is in fact a nonnegative function there, which is the conclusion we wanted to reach. 

\section{Teaser}\label{s:otto}

As already mentioned in the introduction, the ``blow-up and partial regularity'' program has been carried by Luis and some other co-authors in all the four problems mentioned in Section \ref{s:due}. We have examined in fairly many details the first of the problems, for the other three the starting points are the papers \cite{Caffarelli-79,Alt-Caffarelli,ACF}. The theory developed by Luis and co-authors covers much more general situations than the model ones mentioned in Section \ref{s:due} and many results extend to the parabolic setting. For an account of many general aspects we refer the interested reader to the book \cite{Caffarelli-Salsa}. 

Obviously the ``subdivision'' of the free boundary into a regular and singular part leaves open the question of:
\begin{itemize}
\item whether the singular part is truly present;
\item if it is present how large it can be;
\item what structural properties of the solutions we can infer at singular points.
\end{itemize}
The pioneering works of Luis already address some of these questions, but much more was proved subsequently and in fact an entire theory of the ``singular'' points seem to be emerging from the works of several authors in recent years. We refer the readers to the book \cite{Bozho} and the surveys \cite{Alessio,Xavi,FR}.

\bibliographystyle{plain}

\end{document}